\documentclass[10pt]{amsart}
\usepackage{amssymb}
\usepackage{epsfig}
\usepackage{url}
\usepackage[T1]{fontenc}
\usepackage{color}
\usepackage{tikz}
\usetikzlibrary{arrows}
\theoremstyle{plain}

\newtheorem{thm}{Theorem}[section]
\newtheorem{cor}[thm]{Corollary}

\newtheorem{rem}[thm]{Remark}
\newtheorem{ques}[thm]{Question}
\newtheorem{conj}[thm]{Conjecture}

\def\bbb{\mathbb}
\def\op{\operatorname}
\renewcommand{\phi}{\varphi}

\newcommand{\N}{\bbb{N}}
\newcommand{\Z}{\bbb{Z}}

\begin{document}

\title[A note on the Diophantine equation $2^{n-1}(2^{n}-1)=x^3+y^3+z^3$]{A note on the Diophantine equation $2^{n-1}(2^{n}-1)=x^3+y^3+z^3$}
\author{Maciej Ulas}

 \subjclass[2010]{}
\thanks{}

\begin{abstract}
Motivated by the recent result of Farhi we show that for each $n\equiv \pm 1\pmod{6}$ the title Diophantine equation has at least two solutions in integers. As a consequence, we get that each (even) perfect number is a sum of three cubes of integers. Moreover, we present some computational results concerning the considered equation and state some questions and conjectures.
\end{abstract}

\maketitle

\section{Introduction}\label{sec1}
Let $n\in\N_{+}$ and put $P_{n}=2^{n-1}(2^{n}-1)$. We say that $N$ is a perfect number if its the sum of proper divisors. In other words, $\sigma(N)=2N$, where $\sigma(N)=\sum_{d|N}d$. Up to the date, we do not know whether there is an odd perfect number. On the other side, as was proved by Euclid, if $N$ is an even perfect number then $N=P_{p}$, where $p$ and $2^{p}-1$ are primes. An early state of research on perfect numbers is presented in the first chapter in Dickson classical book \cite{Dic}. We know that there are at least 49 even perfect numbers. The largest know corresponds with $p=74207281$. One among many interesting properties of perfect numbers, is the property observed by Heath, that each even perfect number $>6$ is a sum of consecutive odd cubes of positive integers. This observation motivated Farhi to ask what is the smallest number $r$ such that each even perfect number $>6$ is the sum of at most $r$ cubes of non-negative integers. In \cite{Farhi}, Farhi proved that $r=5$ does the job. In fact, he observed that if $n\equiv 1\pmod{6}$, then $M_{n}$ is the sum of three cubes of positive integers. This is simple consequence of the classical polynomial identity
$$
2t^6-1=(t^2+t-1)^3+(t^2-t-1)^3+1.
$$
Indeed, multiplying it by $t^6$ and then taking $t=2^{n}$ we immediately get the representation of $P_{6n+1}$ as sum of three positive cubes. In case of $n\equiv 5\pmod{6}$ the number $P_{n}$ is a sum of five positive cubes. It is important to note that $P_{n}$ is not necessarily perfect in the proof presented by Farhi. Let us also note that perfect numbers corresponding to $p=3, 5, 7, 13, 17$ can be represented as a sum of three cubes of positive integers. This observation motivated Farhi to state the conjecture saying that each perfect number is such a sum (Conjecture 2 in \cite{Farhi}). Unfortunately, we were unable to prove this statement. This fail is a good motivation to consider the Diophantine equation
\begin{equation}\label{maineq}
P_{n}=x^3+y^3+z^3
\end{equation}
for fixed $n$, and asks about its solutions in (not necessarily positive) integers.

The question about the existence of integer solutions of the equation $N=x^3+y^3+z^3$ is a classical one. The equation has no solutions for $N\equiv \pm 4\pmod{9}$ and it is conjectured that there are infinitely many solutions otherwise. However, this conjecture is proved only for $N$ being a cube or twice a cube (see for example \cite{Mor}). It is clear that the number $P_{n}$ is not a cube nor twice a cube and $P_{n}\not\equiv \pm4\pmod{9}$ for all $n\in\N_{+}$. Thus, the question concerning the existence of integer solutions of the equation (\ref{maineq}) is non-trivial.

In Section \ref{sec2} we prove that for $n\equiv 1, 2, 4, 5\pmod{6}$ the Diophantine equation (\ref{maineq}) has at least one solution in integers. Moreover, in the case of $n\equiv \pm 1\pmod{6}$ we show the existence of at least two solutions. We also prove that for each $n\in\N_{+}$ the number $P_{n}$ can be represented as a sum of four cubes of integers. In Section \ref{sec3} we present results of our numerical computations concerning the equation (\ref{maineq}). In particular, for each $n\leq 40$ a solution of (\ref{maineq}) is found and the table of all non-negative solutions is presented. Moreover, we state some questions and conjectures which may stimulate further research.

\section{The results}\label{sec2}
% We will use the following polynomial identities:
%\begin{align*}
%\frac{1}{8}t^6(t^6-2)&=\left(\frac{1}{2}t(t^3+t^2+1)\right)^3+\left(\frac{1}{2}t(t^3-t^2-1)\right)^3-\left(\frac{1}{2}t^4+t^2\right)^3,\\
%          t^6(2t^6-1)&=(t^{4})^3+(t^4)^3-(t^2)^3,\\
%                t^6-2&=(t^2-t-1)^3+(t^2+t-1)^3-(t^2)^3.
%\end{align*}

We have the following
\begin{thm}\label{three}
If $n\equiv 1\pmod{3}$ or $n\equiv 2\pmod{6}$ then the Diophantine equation {\rm (\ref{maineq})} has at least one solution in integers. Moreover, if $n\equiv \pm 1\pmod{6}$ then the Diophantine equation {\rm (\ref{maineq})} has at least two solutions in integers.
\end{thm}
\begin{proof}
Our result is an immediate consequence of the following identities which hold for all $n\in\N_{+}$: \begin{align*}
P_{3n+1}=&(2^{2n})^3+(2^{2n})^3-(2^{n})^3,\\
P_{6n+2}=&(2^{4n+1})^3-(2^{2n})^3-(2^{2n})^3,\\
P_{6n+1}=&(2^{n-2}(2^{3n+2}-21))^{3}+(2^{n-2}(2^{3n+2}+21))^{3}-(11\cdot 2^{2n-1})^3,\\
P_{6n+5}=&(2^{n}(2^{3(n+1)}+2^{2(n+1)}+1))^{3}+(2^{n}(2^{3(n+1)}-2^{2(n+1)}-1))^{3}-(2^{2(n+1)}(2^{2n+1}+1))^3\\
        =&(2^{2n+1}(2^{2(n+1)}-2^{n+1}-1))^{3}+(2^{2n+1}(2^{2(n+1)}+2^{n+1}-1))^{3}-(2^{4n+3})^{3}.
\end{align*}
Replacing $n$ by $2n$ in the first equality we get the second solution of the equation $P_{6n+1}=x^3+y^3+z^3$.
\end{proof}

\begin{rem}
{\rm Let us note that the expression for $P_{6n+1}$ from the proof of Theorem \ref{three}, can be deduced from the polynomial identity
$$
64t^3(2t^6-1)=(4t^3-21)^3+(4t^3+21)^3-(22t)^3
$$
by multiplying both sides by $\frac{1}{64}t^3$, and then taking $t=2^{n}$. Moreover, the first expression for $P_{6n+5}$ follows from the identity
$$
t^3(t^6-2)=(t^3+t^2+1)^3+(t^3-t^2-1)^3-(t(t^2+2))^3
$$
by multiplying both sides by $\frac{1}{8}t^3$, and then taking $t=2^{n+1}$.
}
\end{rem}

\begin{cor}\label{perfect}
For each perfect number $N$, the number of representations of $N$ as a sum of three cubes of integers is $\geq 2$.
\end{cor}
\begin{proof}
From Theorem \ref{three}, we know that for each odd prime $p>3$, the number $N=P_{p}$ has at least two representations as a sum of three cubes of integers. For $p=2, 3$ we have
$$
P_{2}=2^3-1^3-1^3=65^3-43^3-58^3, \quad P_{3}=3^3+1^3=14^3+13^3-17^3,
$$
and get the result.
\end{proof}

We firmly believe that the equation (\ref{maineq}) has solution in integers for each $n\in\N_{+}$ (see Conjecture  \ref{conj1}). Unfortunately, we were unable to prove such statement. Instead, we offer the following

\begin{thm}\label{four}
For each $n\in\N_{+}$, the number $P_{n}$ can be represented as a sum of four cubes of integers.
\end{thm}
\begin{proof}
Let us note the classical identity
$$
t^3 - 2(t-1)^3 + (t-2)^3=6(t-1),
$$
and observe that $P_{2n}\equiv 0\pmod{6}$. Thus, by taking
$$
t=\frac{1}{3}(2^{2(2n-1)}-2^{2(n-1)}+3)
$$
we get the representation of the number $P_{2n}$ as a sum of four cubes.

In order to represents $P_{2n+1}$, we note the identity
$$
(3t-12)^3-(3t-13)^3-t^3+(t-9)^3=2(9t-130).
$$
Using simple induction, we easily get the congruence $P_{2n+1}\equiv 10\pmod{18}$ for $n\in\N_{+}$. Thus, by taking
$$
t=\frac{1}{9}(2^{4n}-2^{2n-1}+130)
$$
we get the representation of the number $P_{2n+1}, n\in\N$, as a sum of four cubes. Our theorem is proved.
\end{proof}

\section{Numerical results, questions and conjectures}\label{sec3}

In order to gain more precise insight into the problem we performed a search for solutions of the equation (\ref{maineq}) in integers. Because we are mainly interested in solutions in non-negative integers we use the following procedure. First of all, let us recall that for $a, b\in\Z$ we have $a^3+b^3\equiv 0, 1, 2, 7, 8\pmod{9}$. Moreover, we observed that the sequence $(P_{n}\pmod{9})_{n\in\N_{+}}$ is periodic of the (pure) period 6. More precisely:
$$
(P_{n}\pmod{9})_{n\in\N_{+}}=\overline{(1, 6, 1, 3, 1, 0)}.
$$
%Thus:
%\begin{equation*}
%\begin{equation*}
%P_{n}-x^3\pmod{9}&=\begin{cases}\begin{array}{ccc}
%                                  0 &  & x\equiv 0\pmod{3}\;\mbox{and}\; n\equiv 0\pmod{6} \\
%                                  0 &  & 2\not |\;\mbox{and}\; x\equiv 1\pmod{3} \\
%                                  2 &  & 2\not |\;\mbox{and}\; x\equiv 2\pmod{3} \\
%
%                                \end{array}
%\end{cases}.
%\end{equation*}
%\end{equation*}

For given $n$ and each $x\in\{0,\ldots, \lfloor P_{n}^{1/3}\rfloor\}$ satisfying $(P_{n}-x^3)\pmod{9}\in\{0,1,2,7,8\}$, we computed the set
$$
D_{n}(x)=\{d\in\N_{+}:\;P_{n}-x^{3}\equiv 0\pmod{d}\},
$$
i.e., the set of all positive divisors of the number $P_{n}-x^{3}$. The congruence condition is useful in some cases because reduce the number of computations which need to be performed. Indeed, if $n\equiv 2, 4\pmod{6}$ then $P_{n}\equiv 6, 3\pmod{9}$ respectively, and we need to have $x\equiv 2\pmod{3}$ ($x\equiv 1\pmod{3}$). Unfortunately, in remaining cases we need to compute all values of $x$ in order to find non-negative solutions. Next, for each $d\in D_{n}(x)$ such that $d<(P_{n}-x^3)/d$, we solved the system of equations
$$
d=y+z,\quad \frac{P_{n}-x^{3}}{d}=y^2-yz+z^2
$$
for $y, z$ and get
$$
y=\frac{1}{6}\left(3d\pm\sqrt{3\left(\frac{4(P_{n}-x^{3})}{d}-d^2\right)}\right),\;z=\frac{1}{6}\left(3d\mp\sqrt{3\left(\frac{4(P_{n}-x^{3})}{d}-d^2\right)}\right).
$$
In consequence, if the numbers $y, z$ computed in this way were integers we got the solution of the equation (\ref{maineq}). This procedure was implemented in Magma computational package \cite{Magma}, and allows us to get all solutions in positive integers of the equation (\ref{maineq}) with $n\leq 40$. The results of our computations are presented in Table 1 below. We also added the value of $g:=\gcd(x,y,z)$.

%One can also apply this procedure to find {\it some} positive solutions of the equation (\ref{maineq}), i.e., for given $n$ choose some $d\in %D_{n}(x)$. In this way one can find the solution $(x,y,z)=(3083584, 32842240, 48722624)$ of the equation (\ref{maineq}) for $n=39$.

\begin{equation*}
\begin{array}{|l|l|l||l|l|l|}
\hline
  n & (x,y,z)              & g          & n & (x,y,z)                        & g\\
  \hline
  3  & (0,1,3)             & 1          & 31 & (1024,1014784,1080320)        & 2^{10}\\
  5  & (4,6,6)             & 2          &    & (53824,684032,1256896)        & 2^6 \\
  7  & (4,4,20)            & 2^2        &    & (90112,464896,1301504)        & 2^{10} \\
  9  & (10,23,49)          & 1          &    & (342016,581120,1274368)       & 2^{9} \\
  11 & (18,94,108)         & 2          &    & (435712,977920,1088000)       & 2^{9}    \\
     & (28,73,119)         & 1          &    & (452624,712312,1227976)       & 2^3\\
  13 & (16,176,304)        & 2^4        &    & (642957,702144,1192051)       & 1\\
  15 & (87,273,802)        & 1          &    & (649984,956288,1049728)       & 2^{7}\\
     & (280,488,736)       & 2^3        & 35 & (103936,1058816,8382976)      & 2^9  \\
  17 & (720,1336,1800)     & 2^3        &    & (825724,2369072,8322436)      & 2^2 \\
  18 & (144,1224,3192)     & 3\cdot 2^3 &    & (1159576,5742485,7364203)     & 1           \\
     & (168,1368,3168)     & 3\cdot 2^3 &    & (1545844,5658327,7401321)     & 1       \\
     & (276,1808,3052)     & 2^2        &    & (2128896,5711872,7332864)     & 2^{10}      \\
     & (968,976,3192)      & 2^3        &    & (2565760,2610912,8220960)     & 2^5          \\
     & (1284,2076,2856)    & 3\cdot 2^2 &    & (4021568,5381152,7175392)     & 2^5          \\
     & (1368,1904,2920)    & 2^3        & 36 & (870912, 8406528, 12088320)   & 3\cdot 2^{9}          \\
  19 & (64,3520,4544)      & 2^6        &    & (3364928, 7935616, 12216768)  & 2^6          \\
     & (1216,1856,5056)    & 2^6        &    & (3663896, 6521760, 12671464)  & 2^3            \\
     & (1968,3516,4420)    & 2^2        & 37 & (4096,16510976,17035264)      & 2^{12}             \\
  21 & (976,9088,11312)    & 2^4        &    & (65536,7086080,20869120)      & 2^{12}          \\
  22 & (13084,14728,14980) & 2^2        &    & (1409488,9313840,20514944)    & 2^{4}           \\
  23 & (10096,19648,29840) & 2^4        &    & (1690048,2408352,21123936)    & 2^{5}          \\
     & (10398,17175,30721) & 1          &    & (1940480,12226048,19669504)   & 2^{9}             \\
     & (19776,20992,26304) & 2^6        &    & (7889536, 14446400, 18109120) & 2^6                \\
  25 & (16,27680,81520)    & 2^4        &    & (2701980,13899489,18889183)   & 1             \\
     & (256,61184,69376)   & 2^8        &    & (5169168,15293424,17894080)   & 2^4  \\
     & (6208,37888,79808)  & 2^6        &    & (5875248,13984848,18669088)   & 2^4  \\
     & (21034,58773,70515) & 1          &    & (10327879,11144196,19091961)  & 1       \\
  26 & (3542,93428,112826) & 2          & 38 & (72704, 24487424, 28477952)   & 2^9             \\
  27 & (39808,89600,201856)& 2^7        & 39 & (3083584, 32842240, 48722624) & 2^6                \\
     & (83110,154196,168298)& 2         &    & (14437236,38893888,44692620)  & 2^2                \\
  28 & (88576,156160,315904) & 2^9      &    & (26259968, 34426624, 45177088)& 2^8                  \\
  29 & (37120,54272,524032)  & 2^8      &    & (29613312,30112512,46079488)  & 2^8                \\
     & (292540,340128,430404)& 2^2      & 40 & (23894752, 58850848, 72873280)& 2^5  \\
  30 & (98816,297216,818944) & 2^8      &    & &                  \\
     & (120576,440992,787808)& 2^5      &    & &                  \\
\hline
\end{array}
\end{equation*}
\begin{center}Table 1. All solutions of the Diophantine equation $P_{n}=x^3+y^3+z^3$ in non-negative integers $x, y, z$ and $n\leq 40$.\end{center}

For given $n$, the time needed to compute solutions with our method was from seconds (for $n\leq 25$) to four days in case of $n=40$. All computations were performed on typical laptop with generation i7 processor and 16 GB of RAM. Moreover, it should be noted that our procedure also computes (some) solutions satisfying $yz<0$, which is a consequence of the construction. In consequence, for each $n\in\{2,\ldots,40\}\setminus\{2, 8, 20\}$, our procedure produce a solution of the equation (\ref{maineq}) with $yz<0$, i.e., exactly one among the numbers $y, z$ is negative. In Table 2 below, we present the integer solution of the equation (\ref{maineq}) without non-negative solutions and with smallest value of $\op{min}\{|x|,|y|,|z|\}$.

\begin{equation*}
\begin{array}{|l|l|c|l|l|}
\hline
  n & (x,y,z)              &  & n & (x,y,z) \\
  \hline
 4  & (-2,4,4)           &  &  24 & (-21716,19656,52340) \\
 10 & (-8,64,64)         &  &  32 & (-5219392,1549376,5285888) \\
 12 & (-54,136,182)      &  &  33 & (-312056,1171940,3280828) \\
 14 & (-430,446,500)     &  &  34 & (-2048,4194304,4194304)\\
 16 & (-32,1024,1024)    &  &     &                             \\
\hline
\end{array}
\end{equation*}
\begin{center}Table 2. Certain integer solutions of the Diophantine equation $P_{n}=x^3+y^3+z^3$ for $n\leq 40$ and without non-negative solutions.\end{center}

Moreover, in Table 3 we present the number of integer solutions which were found by our procedure.
\begin{equation*}
\begin{array}{|c|ccccccccccccc|}
\hline
  n & 2 & 3 & 4 & 5 & 6 & 7 & 8 & 9 & 10 & 11 & 12 & 13 & 14 \\
   \hline
    & 0 & 1 & 1 & 3 & 2 & 2 & 0 & 3 & 2  & 8  & 2  & 6  & 1 \\
   \hline
 \hline
  n & 15 & 16 & 17 & 18 & 19 & 20 & 21 & 22 & 23 & 24 & 25 & 26 & 27 \\
   \hline
    & 4  & 1  & 8  & 38 & 17 & 0  & 7  & 3  & 18 & 4  & 18 & 4  & 16 \\
   \hline
 \hline
  n & 28 & 29 & 30 & 31 & 32 & 33 & 34 & 35 & 36 & 37 & 38 & 39 & 40\\
   \hline
    & 4  & 12 & 11 & 17 & 1  & 4  & 6  & 54 & 14 & 75 & 3  & 10 & 3  \\
  \hline
\end{array}
\end{equation*}
\begin{center}Table 3. The number of integer solutions of the Diophantine equation $P_{n}=x^3+y^3+z^3$, $n\leq 40$, founded by the described procedure.\end{center}

%\begin{equation*}
%\begin{array}{|c|cccccccccccccccccccc|}
%  \hline
%  n & 2 & 3 & 4 & 5 & 6 & 7 & 8 & 9 & 10 & 11 & 12 & 13 & 14 & 15 & 16 & 17 & 18 & 19 & 20 & 21 \\
%  \hline
%    & 0 & 1 & 1 & 3 & 2 & 2 & 0 & 3 & 2  & 8  & 2  & 6  & 1  & 4  & 1  & 8  & 38 & 17 & 0  & 7 \\
% \hline
% \hline
%  n & 22 & 23 & 24 & 25 & 26 & 27 & 28 & 29 & 30 & 31 & 32 & 33 & 34 & 35 & 36 & 37 & 38 & 39 & 40 & \\
%  \hline
%    & 3  & 18 & 4  & 18 & 4  & 16 & 4  & 12 & 11 & 17 & 1  & 4  & 6  & 54 & 14 & 36 & 3  & 10 & 1  &\\
%\hline
%\end{array}
%\end{equation*}
%\begin{center}Table 3. The number of integer solutions of the Diophantine equation $P_{n}=x^3+y^3+z^3$, $n\leq 40$, founded by the described %procedure.\end{center}

The search of solutions for $n=2, 8, 20$ was performed in a similar way, but without the assumption of positivity of $P_{n}-x^{3}$ and with the replacement of $P_{n}-x^{3}$ by $|P_{n}-x^{3}|$. In this way, for $n=2$, we found the solutions of the equation (\ref{maineq}) presented in the proof of Corollary \ref{perfect}. Moreover, we get the equalities
\begin{align*}
P_{8} &=32^3-4^3-4^3=404^3-124^3-400^3,\\
P_{20}&=8192^3-64^3-64^3=9404^3 - 472^3 - 6556^3,
\end{align*}
which fill the gap.

\begin{rem}
{\rm Let us also note that the non-negative solutions of the equation (\ref{maineq}) for given $n$ often satisfy the condition $\gcd(x,y,z)=2^k$ for certain, not to small, value of $k$. Having in mind this property, we performed numerical search of positive solutions for certain values of $n>40$. The method employed was the same as in the case $n\leq 40$, but instead to work for given $n$, with $P_{n}$ we worked with the (smaller) number $M_{k,n}=2^{a_{n}}2^{3k}(2^{n}-1)$, where $k\in\{1,2,3,4,5\}$ and  $a_{n}\equiv n-1\pmod{3}$. Each representation of $M_{k,n}$ after multiplication by $2^{3m}$, where $m=(n-1-a_{n}-3k)/3$, leads to the representation of $P_{n}$ as a sum of three cubes. Using this approach we found the following representations
\begin{align*}
P_{41}&=(2^{12}\cdot 441)^3 +(2^{12}\cdot 22063)^3 +(2^{12}\cdot 29022)^3,\\
P_{42}&=(2^9\cdot 183840)^3+(2^9\cdot 301469)^3+(2^9\cdot 337507)^3,\\
P_{43}&=(2^{14})^3+(2^{14}\cdot 16255)^3+(2^{14}\cdot 16511)^3,\\
P_{45}&=(2^{12}\cdot 18326)^3 + (2^{12}\cdot 144043)^3 + (2^{12}\cdot 181837)^3,\\
P_{47}&=(2^{14}\cdot 5835)^3 + (2^{14}\cdot 41149)^3 + (2^{14}\cdot 129702)^3,\\
P_{48}&=(2^{14}\cdot 8479)^3 + (2^{14}\cdot 160641)^3 + (2^{14}\cdot 169400)^3,\\
P_{49}&=(2^{16})^3+(2^{16}\cdot 65279)^3+(2^{16}\cdot 65791)^3,\\
P_{51}&=(2^{15}\cdot 91838)^3 +(2^{15}\cdot 252707)^3 + (2^{15}\cdot 380629)^3.\\
\end{align*}

}
\end{rem}
Our numerical search and Theorem \ref{three} suggest the following

\begin{conj}\label{conj1}
For each $n\in\N_{+}$ the Diophantine equation {\rm (\ref{maineq})} has a solution in integers.
\end{conj}

From our table we note that the equation (\ref{maineq}) has no solutions in non-negative integers $x, y, z$ for
$$
n=2, 4, 6, 8, 10, 12, 14, 16, 20, 24, 32, 33.
$$
This numerical observation lead us to the following

\begin{conj}\label{conj2}
For each $\epsilon\in\{0,1\}$, there are infinitely many $n\equiv \epsilon \pmod{2}$ such that the equation {\rm (\ref{maineq})} has no solutions in non-negative integers $x, y, z$.
\end{conj}

Moreover, according to our numerical search, one can also ask whether the conjecture proposed by Farhi is not too optimistic. Indeed, in his proof of the existence of representations of a perfect number $P_{p}$ as a sum of five non-negative cubes, with $p\geq 3$, he used only the fact that $p\equiv \pm 1\pmod{6}$ and the well-known polynomial identity
$$
2t^6-1=(t^2+t-1)^3+(t^2-t-1)^3+1,
$$
i.e., any special property of perfect numbers was used. We also observed that the smallest odd $n\in\N_{\geq 3}$, such that the equation (\ref{maineq}) has no solutions in positive integers is 33. Due to our limited experimental data ($n\leq 40$ in our search), there is no strong reason to believe that for all perfect numbers $P_{p}$, the equation $P_{p}=x^3+y^3+z^3$ has a solution in non-negative integers. On the other side, the first possible candidate for the counterexample to the conjecture is $p=89$. The corresponding perfect number $P_{89}$ has 54 digits, and the question about the existence of positive integer solutions of the equation $P_{89}=x^3+y^3+z^3$ is rather difficult

It is also interesting to note the equalities
$$
P_{3}=1^3+3^3, \quad P_{7}=28^3-24^3,\quad P_{9}=60^3-44^3,
$$
which give all solutions of the equation $P_{n}=x^3+y^3, n\leq 140$, in integers. This observation lead us to the following

\begin{ques}
Is the set of integer solutions (in variables $n, x, y$) of the Diophantine equation $P_{n}=x^3+y^3$ finite?
\end{ques}

We expect that the answer is YES.

\bigskip

\noindent Jagiellonian University, Faculty of Mathematics and Computer Science, Institute of Mathematics, {\L}ojasiewicza 6, 30 - 348 Krak\'{o}w, Poland;

\noindent email: {\tt maciej.ulas@uj.edu.pl}

 \end{document}